\documentclass[12pt]{article}
\usepackage{times}
\usepackage{amsfonts}
\usepackage{amsmath}
\usepackage{amssymb}
\usepackage{amsthm}
\usepackage{mathtools}
\usepackage{tikz-cd}
\usepackage{xcolor}
\usepackage{hyperref}
\hypersetup{
	colorlinks=true,
	linkcolor=blue,
	urlcolor=blue,
    citecolor=blue,
    anchorcolor=blue,
	linktoc=all
}
\usepackage{bbm}

\newcommand{\R}{\mathbb{R}}
\newcommand{\C}{\mathbb{C}}

\newcommand{\Z}{\mathbb{Z}}

\newcommand{\MU}{\mathbf{MU}}
\newcommand{\KU}{\mathbf{KU}}
\newcommand{\T}{\mathbb{T}}
\newcommand{\Zp}{\mathbb{Z}_{(p)}}
\newcommand{\A}{\mathbb{A}}

\newtheorem{theorem}{Theorem}[section]
\newtheorem{proposition}[theorem]{Proposition}
\newtheorem{definition}[theorem]{Definition}
\newtheorem{lemma}[theorem]{Lemma}
\newtheorem{corollary}[theorem]{Corollary}

\newtheorem{conjecture}[theorem]{Conjecture}

\date{}
\author{Noah Wisdom}
\title{Properties and Examples of $A$-Landweber Exact Spectra}

\begin{document}
\maketitle

\begin{abstract}
It is classically known that Landweber exact homology theories (complex oriented theories which are completely determined by complex cobordism) admit no nontrivial phantom maps. Herein we propose a definition of $A$-Landweber exact spectra, for $A$ a compact abelian Lie group, and show that an analogous result on phantom maps holds. Also, we show that a conjecture of May on $KU_G$ is false. We do not prove an equivariant Landweber exact functor theorem, and therefore our result on phantom maps only applies to $MU_A$, $KU_A$, their $p$-localizations, and $BP_A$, which are shown to be $A$-Landweber exact by ad-hoc methods.
\end{abstract}

\section{Introduction}

Any complex-oriented ring spectrum gives rise to a formal group law. Conversely, given a formal group law, we can ask if it arises from a complex oriented ring spectrum. In \cite{Lan73} and \cite{Lan76}, Landweber gives checkable criteria on a formal group law which imply that it arises from a complex oriented homology theory. Such a homology theory is said to be Landweber exact; a natural consequence of Landweber's results is that Landweber exact homology theories are completely determined by $MU_*(-)$.

Historically, Landweber exact ring spectra, such as $MU$, $KU$, $BP$, and the Lubin-Tate theories, have played a central role in chromatic homotopy theory. First observed by Hopkins, one important property of Landweber exact ring spectra is that they have no nontrivial phantom maps: every map of spectra which induces the zero natural transformation on homology theories is actually nullhomotopic (cf Lurie's lecture notes \cite{Lur10}).

We take the perspective that the defining property of Landweber exact spectra is that their homology theories are determined by that of $MU$. Namely, there is a natural isomorphism \[ E_*(-) \cong MU_*(-) \otimes_{MU_*} E_* \] induced by a ring map $MU_* \rightarrow E_*$. This perspective straightforwardly generalizes to $A$-spectra with $MU_A$ replacing $MU$, for $A$ a compact abelian Lie group. Surprisingly, we do not need to replace our ring map $MU_* \rightarrow E_*$ with a map of Tambara functors, we only need a ring map $\pi_*^A(MU_A) \rightarrow \pi_*^A(E_A)$. Similarly yet still surprisingly, the $RO(A)$-graded structure plays a very minimal role.

It is the main objective of this paper to argue that this is the correct perspective on equivariant Landweber exactness. In future work, the author hopes to develop algebraic criteria, analogous to Landweber's condition that the Hasse invariants form a regular sequence, which ensure that an equivariant formal group law arises from an equivariant complex oriented theory.

Over the past twenty-five years, the machinery of equivariant complex orientations and equivariant formal group laws has been developed, and foundational results from the nonequivariant story have been ported over to the equivariant world for compact abelian Lie groups. The definition and first properties of equivariant formal group laws were established by Cole, Greenlees, I. Kriz \cite{CGK00}, and Strickland \cite{Str11}. Meanwhile, Sinha \cite{Sin01}, Strickland \cite{Str01}, Abram and Kriz \cite{AK16}, and Kriz and Lu \cite{KL21} studied the homotopy groups of $MU_A$. Recently, Hausmann \cite{Hau22} proved an equivariant version of Quillen's theorem on the Lazard ring, and Hausmann and Meier \cite{HM23} classified the invariant prime ideals of $( (MU_A)_*, (MU_A)_* MU_A )$. In particular, $MU_A$ has played a central role in what could be seen as the foundations of chromatic equivariant homotopy theory. 

With this language, we can ask if a map classifying an $A$-equivariant formal group law determines a homology theory on compact $A$-spaces or $A$-spectra. If so, the resulting representing $A$-spectrum $E$ is called $A$-Landweber exact, and we have the natural isomorphism \[ E_*(-) \cong (MU_A)_*(-) \otimes_{(MU_A)_*} E_* \] Next, we can deduce properties of so-called $A$-Landweber exact theories. In particular, we will show that $A$-Landweber exact theories have no nontrivial phantom maps, essentially by adapting the proof in Lurie's lecture notes \cite{Lur10}. From the definition, $MU_A$ is tautologically $A$-Landweber exact.

Similarly to $MU_A$, our second example of an $A$-Landweber exact spectrum, $KU_A$, has been the object of much study. Originally constructed by Atiyah and Segal \cite{AS66}, a connective variant was constructed by Greenlees \cite{Gre05}, and a global version oriented by $MSpin^c$ was constructed by Joachim \cite{Joa04} and contextualized by Schwede \cite{Sch18}. There is a completion theorem relating $KU_A$ with the Borel completion of $KU$ \cite{AS69} \cite{AHJM88} which is important in the study of characteristic classes. The geometric fixed points of $KU_A$ were studied by tom Dieck in \cite{tom79}, and this study was used by Balderrama to study the equivariant Bousfeld class of $KU_A$ \cite[Section A.4]{Bal22}. We will show via ad-hoc methods that $KU_A$ is $A$-Landweber exact. More precisely, we will identify the equivariant Conner-Floyd isomorphism of Okonek \cite{Oko82} with the natural transformation arising from equivariant complex orientability of $KU_A$.

Our final example is equivariant Brown-Peterson theory. Originally constructed by \cite{BP66}, constructions of a completely different flavor have been given by Quillen \cite{Qui69}, Priddy \cite{Pri80}, and by Baas-Sullivan theory \cite{Baa73}. Perhaps most importantly, $BP$ has played a central role in Adams-Novikov spectral sequence computations of the stable homotopy groups of spheres by Miller, Ravenel, and Wilson \cite{MRW77}. Equivariant versions of $BP$ have not enjoyed as much study as $MU_A$ and $KU_A$, although a construction due to May \cite{May98} is known. We will give an alternate construction of $BP_A$ manifestly as an $A$-Landweber exact cohomology theory. Our construction is most analogous to Quillen's construction; it would be very interesting to know if an analogue of Priddy's cellular construction of $BP$ goes through to construct $BP_A$.

It is an artifact of the methods of this paper that we are often able to state things in terms of cohomology rather than homology. This choice is forced upon us by the use of cohomological language throughout \cite{Oko82}. An equivariant Landweber exact functor theorem will most likely require homological language due to that fact that the Hopf algebroid $( (MU_A)_*, (MU_A)_*(MU_A) )$ which classifies equivariant formal group laws and their isomorphisms is homological rather than cohomological. Therefore we attempt to mantain bridges from cohomological language to homological language, cf Proposition \ref{Bridges}.

The contents of this paper are as follows. In section 2, we propose a definition of $A$-Landweber exactness and show that $MU_A$ and $KU_A$ are $A$-Landweber exact. Also, we use the equivariant Quillen theorem to resolve a conjecture of May concerning $MU_G \wedge_{MU} KU$ (for all compact Lie groups $G$) from \cite{May98}. In section 3, we construct $BP_A$, show that it is $A$-Landweber exact, and identify it with May's construction $MU_A \wedge_{MU} BP$. In section 4, we show that any phantom map between equivariant Landweber exact spectra is nullhomotopic.

\section{Equivariant Landweber exactness}

The main goal of this section is to give a definition of equivariant Landweber exactness. Our definition is motivated by the fact that $MU_A$ and $KU_A$ ought to be equivariantly Landweber exact. In particular, nonequivariantly, Landweber exactness of $KU$ is recovered by the Conner-Floyd isomorphism, which historically predates Landweber's results. Okonek has already generalized the Conner-Floyd isomorphism to an equivariant version for compact abelian Lie groups, and we use this as our starting point.

Our secondary goal in this section is to show that the map $(MU_A)^* \rightarrow (KU_A)^*$ defining the Conner-Floyd isomorphism is the same as the map classifying the equivariant formal group law arising from the equivariant complex orientation of $KU_A$. To show this, we will utilize the global homotopy theory of Schwede \cite{Sch18} and Bohmann \cite{Boh12}. We will also need results of Hausmann \cite{Hau22} and Hausmman-Meier \cite{HM23}, which we review here.

It is worth noting that we only need to know that $KU_A$-cohomology is determined by $MU_A$-cohomology in order to show that it has no nontrivial phantom maps; we do not need to know that it is specifically the equivariant formal group law data which specifies this determination. However, this perspective unifies this example with the example of $BP_A$, and omitting this language would unnecessarily obfuscate the component ideas. Furthermore, most nonequivariant spectra known to be determined by $MU$-cohomology in the way which feeds into the phantom map argument are constructed by supplying carefully chosen formal group laws into the Landweber exact functor theorem.

\subsection{Review of global methods}

Given an orthogonal spectrum $X$, we can give $X$ the trivial $A$-action, and thus define the $A$-equivariant homotopy groups $\pi_*^A$ of $X$. The $\mathcal{F}$-global homotopy category with respect to a family of groups $\mathcal{F}$ (which for us will be the family of compact abelian Lie groups) is constructed as the localization of the category of orthogonal spectra with respect to the class of maps which induce isomorphisms on $\pi_*^A$ for each $A$ in the family $\mathcal{F}$.

Schwede \cite[Section 6.1]{Sch18} gives a construction of a global commutative ring spectrum $\MU$, and Joachim \cite{Joa04} constructs (without the language of global homotopy theory) a global commutative ring spectrum $\KU$. Respectively, these recover tom Dieck's homotopical cobordism $MU_A$ and the equivariant $K$-theory $KU_A$ of Atiyah and Segal.

The main property of global spectra which we will require is the following: for a global spectrum $E$, the homotopy groups $\pi_*^A(E_A)$ assemble to a contraviant functor from the category $Ab$ of compact abelian Lie groups to graded abelian groups. If $E$ is a commutative ring spectrum, then $\underline{\pi}_*(E)$ is a contravariant functor from $Ab$ to the category of (graded-commutative) rings.

\begin{definition}[Hausmann \cite{Hau22}]
A complex orientation of a global commutative ring spectrum $E$ is an $RO(\mathbb{T})$-graded unit $t \in \pi_{2-\tau}^\mathbb{T}(E)$ which restricts to $1$ at the trivial group. Here $\tau$ is the tautological complex representation of $\mathbb{T}$.
\end{definition} 

By an argument in \cite[Section 6.1]{Sch18}, $\MU$ is equipped with Thom classes which specify a global complex orientation in this sense. 

Furthermore, we can see that $\KU$ is complex oriented as follows. The Bott classes $\beta_{\mathbb{T},\tau}$ and $\beta_{\mathbb{T},2}$ of \cite[Construction 6.3.46]{Sch18} are $RO(\mathbb{T})$-graded units in $\pi_{- \tau}^\mathbb{T}(KU_\mathbb{T})$ and $\pi_{-2}^\mathbb{T}(KU_\mathbb{T})$ respectively. In particular, $\beta_{\mathbb{T},\tau}^{-1} \beta_{\mathbb{T},2}$ is an $RO(\mathbb{T})$-graded unit in $\pi_{2 - \tau}(KU_\mathbb{T})$. By \cite[Theorem 6.9]{Joa04}, the restriction of this element to the trivial group $\{ 1 \}$ is $\beta_{1,2}^{-1} \beta_{1,2} = 1$.

Global complex orientations give rise to an algebraic structure on global homotopy groups similarly to the way in which a nonequivariant complex orientation gives rise to a formal group law. For technical reasons, we have to restrict our global homotopy groups from all compact abelian Lie groups to tori, although this turns out to not be an issue in the cases of interest.

\begin{definition}[Hausmann \cite{Hau22}]
A global group law is a functor $X : Tori^{op} \rightarrow Rings$ equipped with an element $e \in X(S^1)$ such that for every split surjective character $V : \mathbb{T}^n \rightarrow \mathbb{T}$, the sequence \[ 0 \rightarrow X(\mathbb{T}^n) \xrightarrow{V^*(e) \cdot -} X(\mathbb{T}^n) \xrightarrow{res_{ker(V)}^{\mathbb{T}^n}} X(ker(V)) \rightarrow 0 \] is exact.
\end{definition}

Given a global group law, we left Kan extend it to a functor $Ab^{op} \rightarrow Rings$ and identify our original functor with this extension. Unfortunately, since not every global ring spectrum has the property that their $A$-homotopy groups for general groups $A$ are left Kan extended from their $A$-homotopy groups as $A$ runs through tori, this means that a priori not every complex oriented global ring spectrum $E$ has $\underline{\pi}_*(E)$ equal to the associated global group law. However, if $\pi_*^A(E)$ is even for each group $A$, not just tori, then $\underline{\pi}_*(E)$ agrees with the left Kan extension of the associated global group law. In particular, this is satisfied by both $\MU$ and $\KU$.

Given a complex orientation of a global spectrum $E$, Hausmann shows that a homotopy orbit construction specifies an $A$-equivariant complex orientation of $E_A$ for each compact abelian Lie group $A$. It is clear from this construction that the complex orientation of $\MU$ gives the standard complex orientation of $MU$, and similarly for $\KU$ and $KU$. 

\begin{theorem}[Hausmann \cite{Hau22}]
$\underline{\pi}_*(\MU)$ is an initial object in the category of global group laws. Furthermore $MU_A^*$ is isomorphic to the $A$-equivariant Lazard ring $L_A$.
\end{theorem}

\begin{corollary}
There is a map of global group laws $\underline{\pi}_*(\MU) \rightarrow \underline{\pi}_*(\KU)$ which classifies the global group law induced by the complex orientation of $\KU$. For each abelian group $A$ this map refines to the map $L_A \cong MU_A^* \rightarrow KU_A^*$ classifying the $A$-equivariant formal group law over $KU_A^*$.
\end{corollary}

For $A$ the trivial group, this recovers the usual map $MU^* \rightarrow KU^*$ classifying the formal group law. This map may be viewed as coming from the complex orientation $MU \rightarrow KU$. Equivalently, this map is the Conner-Floyd map, which induces the Conner-Floyd isomorphism \[ KU^*(X) \cong MU^*(X) \otimes_{MU^*} KU^* \] In particular, this gives a proof that $KU$ is Landweber exact, and we take this as our starting point towards a definition of $A$-Landweber exactness.

\subsection{Weakly global cohomology theories with Thom classes}

Recall that Okonek has proven \cite{Oko82} that a certain natural transformation $MU_A(-) \rightarrow KU_A(-)$ induces the equivariant Conner-Floyd isomorphism \[ KU_A^*(X) \cong MU_A^*(X) \otimes_{MU_A^*} KU_A^* \] for all compact abelian Lie groups $A$ and compact $G$-spaces $X$. It ought to be true that $KU_A$ is $A$-equivariantly Landweber exact, and it is reasonable to expect that this exactness is realized by the equivariant Conner-Floyd isomorphism. 

Before we define $A$-Landweber exactness, we will interpret the equivariant Conner-Floyd isomorphism in terms of equivariant formal group laws. This connection is provided by the following observation: the map $MU_A^* \rightarrow KU_A^*$ inducing the $A$-equivariant Conner-Floyd isomorphism is the same as the map which classifies the $A$-equivariant formal group law over $KU_A^*$.

In order to prove this, we will take the global approach. In particular, we will need to know that these maps commute with restrictions along group homomorphisms. Then we may proceed in two steps. First, we can boost the identification of these maps at the trivial group to an identification for all tori. Second, since the values of $\underline{\pi}_*(\MU)$ and $\underline{\pi}_*(\KU)$ at a general compact abelian Lie group $A$ are left Kan extended from the values at tori, we get the result for general $A$.

If we already knew the proposition, then the Conner-Floyd maps would assemble to a map of $Ab$-algebras. Okonek observes in \cite[Lemma 1.6]{Oko82} that the Conner-Floyd map at level $A$ arises from the fact that $MU_A^*(-)$ gives the universal example of an $A$-equivariant cohomology theory with Thom classes. So, it is natural to expect that $\MU$ gives the universal example of a "weakly global cohomology theory with Thom classes". Stated properly, this is true, and it is why the Conner-Floyd maps are global. First, we restate a definition from \cite{Oko82}.

\begin{definition}[Okonek Definition 1.2 {Oko82}]
A Thom class for an equivariant complex stable cohomology theory $E_G^*(-)$ with stability-suspension isomorphisms \[ \sigma : E^*(-) \cong E^{*+dim_\C(V)}(S^V \wedge -) \] is a collection of elements $\tau(p)$, such that for each complex $G$-vector bundle $p : E^k \rightarrow X$, $\tau(p) \in E_G^{2k}(M(p))$, satisfying
\begin{enumerate}
\item (Naturality). For a bundle map $f : p \rightarrow q$, $\tau(p) = (Mf)^* \tau(q)$.
\item (Multiplicativity). $\tau(p \times q) = \tau(p) \wedge \tau(q)$.
\item (Normalization). The Thom class of a $G$-representation $V$ is $\tau(V) = \sigma(V)(1)$.
\end{enumerate}
\end{definition}

Note that if $E$ is complex stable, then we can suspend $1 \in E^0$ to an element of $E^V(S^V)$, apply complex stability to obtain the Thom class $\tau_V \in E^{dim_\C(V)}$, then pullback along the inclusion of $0$ and $\infty$, $S^0 \hookrightarrow S^V$, to obtain an element of $E^{dim_\C(V)}$ which Okonek \cite{Oko82} calls the Euler class. Note that we only needed complex stability for this definition. Euler classes defined this way appear in \cite{CGK00} and are ultimately the same Euler classes appearing in \cite{Hau22}.

\begin{definition}
A global spectrum $E$ is said to determine a weakly global cohomology theory with Thom classes if for each group $A$, there are Thom classes $\tau_A(p)$ for the $A$-equivariant cohomology theory $E_A^*(-)$ in the sense of the above definition, which collectively satisfy $res_B^A(\tau_B(p)) = \tau_A(\alpha^*(p))$ for every group homomorphism $\alpha : A \rightarrow B$ and complex $B$-vector bundle $p$.
\end{definition}

\begin{proposition}
The cohomology theories determined by $\MU$ along with their Thom classes as defined in \cite{Oko82} assemble to determine the initial weakly global cohomology theory with Thom classes, in the sense that the natural transformations of \cite[Lemma 1.6]{Oko82} commute with restriction maps.
\end{proposition}

\begin{proof}
First we must verify that $\MU$ determines a weakly global cohomology theory with Thom classes. This follows by straightforwardly modifying \cite[Construction 6.1.18]{Sch18}, which implies that $\mathbf{MO}$ determines a weakly global cohomology theory with Thom classes for vector bundles over $\R$ instead of over $\C$.

Now let $E$ determine a weakly global cohomology theory with Thom classes $\tau_G$. Okonek \cite[Lemma 1.6]{Oko82} states that for each group $B$ there is a unique natural, stable, and multiplicative transformation $T_B : MU_B^*(-) \rightarrow E_B^*(-)$ such that $T_B(t_B(\gamma(V))) = \tau_B(\gamma(V))$ for each $B$-representation $V$, where $\gamma(V)$ is the canonical bundle over $Gr(V)$.

Let $\alpha : A \rightarrow B$ be any group homomorphism and $z \in MU_B^*(X)$ an arbitrary element represented by some map $f : S^V \wedge X \rightarrow M(\gamma(V))$ into the Thom space of an appropriate canonical $B$-vector bundle. Letting $\sigma(V)$ denote the suspension isomorphisms, we then have \[ \alpha^* \circ T_B(z) = \alpha^* \left( \sigma(V)^{-1} f^* \tau_B(\gamma(V)) \right) \] \[ = \sigma(\alpha^*(V))^{-1} (\alpha^* f)^* \tau_A(\alpha^*(\gamma(V))) \] \[ = \sigma(\alpha^*(V))^{-1} (\alpha^* f)^* \tau_A(\gamma(\alpha^*(V))) \] where the first equality comes from the construction of $T_B$ and the second equality comes from compatibility of restriction maps with suspension isomorphisms, naturality of restriction maps, and compatibility of restriction maps and Thom classes.

Note that $\alpha^*(z)$ is represented by $\alpha^*(f) : S^{\alpha^*(V)} \wedge \alpha^*(X) \rightarrow \alpha^* M(\gamma(V)) = M(\gamma(\alpha^*(V)))$ so we get \[ T_A \circ \alpha^*(z) = \sigma(\alpha^*(V))^{-1} (\alpha^* f)^* \tau_A (\gamma(\alpha^*(V))). \] This shows that the $T_B$ commute with restriction maps.
\end{proof}

Once we know that $\KU$ determines another example of a weakly global cohomology theory with Thom classes, this proof furthermore shows that the associated universal natural transformation is given by the equivariant Conner-Floyd natural transformation at each level.

\begin{proposition}
$\KU$ determines a weakly global cohomology theory with Thom classes which at level $A$ is the equivariant $K$-theory with Thom classes as stated in \cite{Oko82}.
\end{proposition}

\begin{proof}
It follows from \cite[Theorem 6.9]{Joa04} that $\KU$ determines a weakly global cohomology theory with Thom classes. Namely, $\KU$ ultimately receives its Thom classes from the Thom classes of a global version of $MSpin^c$. Directly from the construction these Thom classes are visibly preserved by restriction maps.
\end{proof}

\begin{corollary}
The equivariant Conner-Floyd transformations $MU_A^* \rightarrow KU_A^*$ commute with all restriction maps.
\end{corollary}

\subsection{$KU_A$ is $A$-Landweber exact}

The following proposition provides an equivariant extension of the fact that the map $MU^* \rightarrow KU^*$ inducing the Conner-Floyd isomorphism is the same map as the formal group law classification map.

\begin{proposition}
For each group $A$, the Conner-Floyd map $MU_A^* \rightarrow KU_A^*$ agrees with the map which classifies the $A$-equivariant formal group law over $KU_A$.
\end{proposition}

The result for $A = \{ 1 \}$ is known classically. Now the strategy is to prove the result for tori, inducting on dimension, and then left Kan extend the result to general groups $A$.

\begin{proof}
We begin by inducting on $n$. Assume that the formal group law classification map $MU_{\mathbb{T}^n}^* \rightarrow KU_{\mathbb{T}^n}^*$ agrees with the Conner-Floyd map. Now let $\alpha : \mathbb{T}^{n+1} \rightarrow \mathbb{T}^1$ be a split surjective character with kernel $\mathbb{T}^n$.

The formal group law classification map preserves Euler classes because they are defined in terms of the formal group law structure. According to \cite{CGK00}, the Euler classes may also be described as coming from the complex stability isomorphism, ie they are the Thom classes of vector bundles over a point. Thus, the Conner-Floyd maps also preserve Euler classes. In particular, both maps agree on Euler classes.

Since the collection of homotopy groups of $\MU$ and $\KU$ form global group laws, we can form the following diagram \[ \begin{tikzcd}
0 \arrow{r} & MU_{\T^{n+1}} \arrow{r}{e_\alpha \cdot (-)} \arrow{d} & MU_{\T^{n+1}} \arrow{r}{Res_{\T^n}^{\T^{n+1}}} \arrow{d} & MU_{\T^n} \arrow{r} \arrow{d} & 0 \\
0 \arrow{r} & KU_{\T^{n+1}} \arrow{r}{e_\alpha \cdot (-)} & KU_{\T^{n+1}} \arrow{r}{Res_{\T^n}^{\T^{n+1}}} & KU_{\T^n} \arrow{r} & 0 \end{tikzcd} \] We may take the vertical maps to be either the Conner-Floyd maps or the equivariant formal group law maps, and in either case the diagram is commutative.

From this, we deduce that any element of $MU_{\T^{n+1}}$ for which the values of the two maps under consideration do not agree must be infinitely divisible by $e_\alpha$. By \cite[Corollary 4.10]{HM23} such an element must be zero. So, the Conner-Floyd maps agree with the formal group law classification maps when $A$ is a torus.

Now let $A$ be any compact abelian Lie group. Then, since $\pi_*^A(\MU)$ and $\pi_*^A(\KU)$ are given by left Kan extending the respective functors $\underline{\pi}_*(\MU), \underline{\pi}_*(\KU) : Tori^{op} \rightarrow Rings$, they are described as colimits \[ \pi_*^A(\MU) = colim_{A \rightarrow \T^n} MU_{\T^n}^* \] \[ \pi_*^A(\KU) = colim_{A \rightarrow \T^n} KU_{\T^n}^*. \] Furthermore, for the ring $MU_{\T^n}^*$ indexed by $\beta : A \rightarrow \T^n$ appearing in the colimit, the canonical map to the colimit $MU_{\T^n}^* \rightarrow MU_A^*$ is the restriction along $\beta$, and similarly for $\KU$. 

The Conner-Floyd maps and the formal group law maps both specify the same map of diagrams $\{ MU_{\T^n} \}_{A \rightarrow \T^n} \rightarrow \{ KU_{\T^n} \}_{A \rightarrow \T^n}$, so they specify the same map from $\{ MU_{\T^n} \}_{A \rightarrow \T^n}$ to the constant diagram at $KU_A^*$. By the universal property of colimits, there is a unique map $MU_A^* \rightarrow KU_A^*$ which extends the diagram map to a map of cones on the diagram. Both the Conner-Floyd map and the formal group law classification map give such an extension, hence these maps are equal.
\end{proof}

\begin{corollary}
The global group law map $\underline{\pi}_*(\MU) \rightarrow \underline{\pi}_*(\KU)$ induces isomorphisms $KU_A^*(-) \cong \MU_A^*(-) \otimes_{MU_A^*} KU_A^*$ of cohomology theories on compact $A$-spaces.
\end{corollary}

\begin{proof}
This follows from the above proposition, combined with Theorem 3.6 of \cite{Oko82}.
\end{proof}

We take this Corollary as our model for the following definitions.

\begin{definition}
An $A$-equivariant formal group law classified by $f: L_A \rightarrow k$ is $A$-Landweber exact if $X(-) := MU_A^*(-) \otimes_{MU_A^*} X(A)$ is a cohomology theory on compact $A$-spaces. A global group law $X$ is weakly globally Landweber exact if for each compact abelian Lie group $A$, the canonical map $L_A \rightarrow X(A)$ determines a cohomology theory as above.
\end{definition}

\begin{definition}
A complex oriented $A$-spectrum is $A$-Landweber exact if its associated $A$-equivariant formal group law is. A complex oriented global spectrum is weakly globally Landweber exact if its associated global group law is.
\end{definition}

\begin{corollary}
$KU_A$ and $MU_A$ are $A$-Landweber exact. $\KU$ and $\MU$ are weakly globally Landweber exact.
\end{corollary}

We use the adjective "weakly" because it is unclear how to recover a cohomology theory on global spaces or spectra from this condition. Hopefully, there is either some method to construct a global spectrum from a collection of $A$-spectra related by restriction maps, or some notion of "strong globally Landweber exact" which more directly implies the existence of a global spectrum. An example of the former would be reconstructing $\mathbf{KU}$ from the $KU_A$. A possible approach for the latter would be a global Mackey functor perspective on global Landweber exactness.

Note that since $MU_A^*$ is concentrated in even degrees \cite{Lof73} \cite{Com96}, so are the homotopy groups of any $A$-Landweber exact spectrum. In particular, the $A$-equivariant homotopy groups of any weakly globally Landweber exact spectrum for general $A$ are left Kan extended from the $A$-equivariant homotopy groups for tori.

\subsection{Observations on $A$-Landweber exact spectra}

Later on it will be useful to convert between the homological and the cohomological perspective. When $A$ is the trivial group, the following proposition connects our cohomological perspective to the more traditional homological perspective on Landweber exactness.

Let $E$ a complex oriented $A$-spectrum with $A$-equivariant formal group law classification map $MU_A^* \rightarrow E^*$ (equivalently $(MU_A)_* \rightarrow E_*)$ with orientation class in degree $2$. By \cite[Theorem 1.2]{CGK02} there is a ring map $MU_A \rightarrow E$ inducing the complex orientation of $E$. This ring map induces natural transformations \[ MU_A^*(-) \otimes_{MU_A^*} E^* \rightarrow E^*(-) \] and \[ (MU_A)_*(-) \otimes_{(MU_A)_*} E_* \rightarrow E_*(-) \] of functors on compact $A$-spectra.

\begin{proposition}
\label{Bridges}
For $E$ a complex oriented $A$-spectrum, (1), (2), and (3) are equivalent, and each implies (4).
\begin{enumerate}
\item $E^*(-) \cong MU_A^*(-) \otimes_{MU_A^*} E^*$ as functors on compact $A$-spectra.
\item $E^*(-) \cong MU_A^*(-) \otimes_{MU_A^*} E^*$ as functors on compact $A$-spaces.
\item $E_*(-) \cong (MU_A)_*(-) \otimes_{(MU_A)_*} E_*$ as functors on compact $A$-spectra.
\item $E_*(-) \cong (MU_A)_*(-) \otimes_{(MU_A)_*} E_*$ as functors on compact $A$-spaces.
\end{enumerate}
\end{proposition}

\begin{proof}
Clearly (1) implies (2), and (3) implies (4). By Spanier-Whitehead duality, (1) and (3) are equivalent. Since every compact $A$-spectrum may be written as a filtered colimit of suspension spectra of compact $A$-spectra and both sides of (1) commute with filtered colimits, (2) implies (1).
\end{proof}

We will use the following uniqueness theorem later to identify our construction of $BP_A$ with May's construction $MU_A \wedge_{MU} BP$ from \cite{May98}.

\begin{theorem}
\label{Uniqueness}
Let $E$ and $E'$ be $A$-equivariant complex oriented spectra with orientation class in degree $2$ and $MU_A \rightarrow E$, $MU_A \rightarrow E'$ the associated ring maps of \cite[Theorem 1.2]{CGK02}. If $E$ is $A$-Landweber exact and $E^*$ and $E'^*$ are isomorphic as rings under $MU_A^*$, then $E$ and $E'$ are weakly equivalent as ring $A$-spectra under $MU_A$.
\end{theorem}

\begin{proof}
By assumption we have a natural transformation of multiplicative cohomology theories \[ E^*(-) \cong MU_A^*(-) \otimes_{MU_A^*} E^* \cong MU_A^*(-) \otimes_{MU_A^*} E'^* \rightarrow E'^*(-) \] which is clearly an isomorphism when the input is the sphere spectrum. 

Let $B \subset A$ a closed subgroup. Then the forgetful functor from $A$-spectra to $B$-spectra carries the $A$-orientations $MU_A \rightarrow E$ and $MU_A \rightarrow E'$ to $B$-orientations $MU_B \rightarrow res_B^A(E)$ and $MU_B \rightarrow res_B^A(E')$. Since $res_B^A(E)$ is $B$-Landweber exact, we may repeat the above argument, concluding that our original natural transformation induces isomorphisms \[ \pi_*^B(E) = \pi_*^B(res_B^A(E)) \cong \pi_*^B(res_B^A(E')) = \pi_*^B(E') \] Therefore this natural tranformation defines a weak equivalence of ring $A$-spectra $E \cong E'$.

Finally, by \cite[Theorem 1.2]{CGK02}, $A$-equivariant complex orientations of $A$-Landweber exact spectra correspond bijectively with maps of ring $A$-spectra from $MU_A$. By construction, our weak equivalence $E \cong E'$ preserves the respective complex orientations, hence is an equivalence under $MU_A$.
\end{proof}

All of the $A$-equivariantly complex oriented spectra in this paper have orientation classes in degree $2$. In particular, this property is enjoyed by any $A$-Landweber exact spectrum, and by any $A$-equivariantly complex oriented spectrum arising from May's construction \cite{May98} via the following proposition.

\begin{proposition}
\label{MayCpxOri}
Let $E$ be a complex oriented ring spectrum whose formal group law is classified by $\phi : MU^* \rightarrow E^*$. Then May's construction $MU_A \wedge_{MU} E$ \cite[Theorem 1.1]{May98} is an $A$-equviariant complex oriented ring $A$-spectrum whose $A$-equivariant formal group law is classified by \[ MU_A^* \cong MU_A^* \otimes_{MU^*} MU^* \xrightarrow{Id \otimes \phi} MU_A^* \otimes_{MU^*} E^* \cong E^* \]
\end{proposition}

\begin{proof}
May's construction $MU_A \wedge_{MU} (-)$ is both functorial for $MU$-modules and monoidal. Thus we obtain a ring map $MU_A \cong MU_A \wedge_{MU} MU \rightarrow MU_A \wedge_{MU} E$. By \cite[Theorem 1.2]{CGK02}, this determines an $A$-equivariantly complex oriented cohomology theory.

According to \cite{May98}, the isomorphisms $MU_A^* \otimes_{MU^*} MU^* \cong \pi_*^A( MU_A )$ and $\pi_*^A ( MU_A \wedge_{MU} E ) \cong MU_A^* \otimes_{MU^*} E^*$ arise ultimately from a K\"{u}nneth spectra sequence. Naturality of this spectral sequence with respect to $MU \rightarrow E$ shows that the induced map is precisely the base change of $\phi$
\end{proof}

\begin{conjecture}
If $E$ is a Landweber exact spectrum, then $MU_A \wedge_{MU} E$ is an $A$-Landweber exact spectrum.
\end{conjecture}

Note that this conjecture is true when $E = BP$ by Theorem \ref{BP_A_Identification}. Conversely, by Corollary \ref{MayConjTori}, $A$-Landweber exactness of $KU_A$ does not prove the conjecture for $E = KU$.

As observed in \cite[Example 5.16]{Hau22}, for any global group law $X$ and ring map $X(1) \rightarrow k$, we obtain a new global group law by base change. Specifically, replace $X(A)$ with $X(A) \otimes_{X(1)} k$. May's construction places us exactly in this situation.

\begin{proposition}
\label{May_GGL}
In the situation of Proposition \ref{MayCpxOri}, the homotopy groups $MU_A^* \otimes_{MU^*} E^*$ assemble to a global group law.
\end{proposition}

\begin{proof}
Clearly $MU_{(-)}^* \otimes_{MU^*} E^*$ receives a map from $\mathbf{L} \cong \underline{\pi}_*(\mathbf{MU})$. It remains to check that we have the desired short exact sequences. In particular, we must show that the short exact sequences \[ 0 \rightarrow MU_{\mathbb{T}^n}^* \xrightarrow{e_V} MU_{\mathbb{T}^n}^* \xrightarrow{res_{ker(V)}^{\mathbb{T}^n}} MU_{ker(V)}^* \rightarrow 0 \] remain exact upon tensoring over $MU^*$ with $E^*$. This is clear, however, since $res_{\mathbb{T}^n}^{ker(V)}$ provides an $MU^*$-module splitting.
\end{proof}

\begin{conjecture} 
\label{Global_Existence}
There is a global version of May's functor $E \mapsto \mathbf{MU} \wedge_{MU} E$ which takes Landweber exact spectra to weakly globally Landweber exact spectra.
\end{conjecture}

\begin{conjecture}
If $E$ is a Landweber exact spectrum, then for any subgroup $B$ of a compact abelian Lie group $A$, the cohomology theory $X \mapsto E^*( \Phi^B(X)_{h(A/B)})$ is $A$-Landweber exact ($\Phi^B$ denotes geometric fixed points and ${}_{h(A/B)}$ denotes homotopy orbits). Cohomology theories of this form include the equivariant Lubin-Tate theories of Strickland \cite[Section 11]{Str11}.
\end{conjecture}

It is a curious artifact of our methods that Mackey functors and Tambara functors do not feature prominently in the definition of $A$-Landweber exactness. Roughly speaking, this is in contrast to the global perspective on equivariant formal group laws, where the restriction maps play a pivotal role. Presumably there is some formulation of $A$-Landweber exactness for which the cohomology theories are viewed as taking values in Mackey functors rather than abelian groups, and it would be interesting to see if there are any $A$-spectra besides $MU_A$ which satisfy the resulting condition. Furthermore, such a formulation may lead straightforwardly to a global perspective on Landweber exactness.

\subsection{Properties of $KU_A$}

It is our aim in this subsection to show that the following conjecture of May is false for all nontrivial compact Lie groups $G$. We will start by examining the conjecture when $G = A$, a compact abelian Lie group. Next, an easy argument extends the result to all compact Lie groups containing a nontrivial compact abelian Lie group, ie all nontrivial compact Lie groups.

\begin{conjecture}[\cite{May98}]
\label{MayConj}
$MU_G \wedge_{MU} KU$ is equivalent to $KU_G$ in the category of $MU_G$-module spectra for $G$ a compact Lie group.
\end{conjecture}

By the Proposition \ref{MayCpxOri}, $MU_A \wedge_{MU} KU$ is $A$-equviariantly complex oriented with homotopy groups $MU_A^* \otimes_{MU^*} KU^*$. Regarding $KU$ as an $MU$-module spectrum by nonequivariant Landweber exactness, we see from \ref{MayCpxOri} that the $A$-equvariant formal group law of $MU_A \wedge_{MU} KU$ is classified by \[ MU_A^* \cong MU_A^* \otimes_{MU^*} MU^* \xrightarrow{Id \otimes \rho} MU_A^* \otimes_{MU^*} KU^* \] where $\rho : MU^* \rightarrow KU^*$ classifies the nonequivariant formal group law over $KU^*$.

Writing $KU^* \cong \Z[\beta^{\pm 1}]$, the formal group law associated to $KU$ is $x+y+\beta xy$. Thus $\beta$ is in the image of the formal group law classification map $MU^* \rightarrow KU^*$. For the rest of the subsection, fix a preimage $\tilde{\beta} \in MU^*$ of $\beta$.

\begin{lemma}
\label{KU_Surj}
The map $MU_A^*[\tilde{\beta}^{-1}] \rightarrow KU_A^*$ induced by the $A$-equivariant formal group law classification map $MU_A^* \rightarrow KU_A^*$ is surjective.
\end{lemma}

\begin{proof}
We have $KU_A^* \cong KU^*[A^*]$, the group ring over $KU^*$ on the Pontryagin dual $A^*$ of $A$. The nontrivial elements of $A^*$ correspond to the nontrivial Euler classes under the isomorphism $KU^*[A^*] \cong KU_A^*$, hence each Euler class of $KU_A^*$ is in the image of the $A$-equivariant formal group law classification map. Inverting $\tilde{\beta} \in MU^* \subset MU_A^*$, the result follows from the case $A = \{ 1 \}$, which follows from $KU^* \cong \Z[\beta^{\pm 1}]$.
\end{proof}

\begin{theorem}
For each $A$, there is exactly one map $MU_A^* \otimes_{MU^*} KU^* \rightarrow KU_A^*$ of rings under $MU_A^*$. This map is always surjective, but only injective when $A = \{ 1 \}$.
\end{theorem}

\begin{proof}
The commutative diagram with horizontal maps the formal group law classification maps 
\[ \begin{tikzcd}
MU^* \arrow{r} \arrow{d}{res_A^{ \{ 1 \} }} & KU^* \arrow{d}{res_A^{ \{ 1 \} }} \\
MU_A^* \arrow{r} & KU_A^*
\end{tikzcd} \]
specify a ring map $MU_A^* \otimes_{MU^*} KU^* \rightarrow KU_A^*$. This ring map must be surjective since $KU^*$ is clearly in the image, and $KU_A^*$ is generated over $KU^*$ by Euler classes \cite{Joa04} \cite[Section 6.4]{Sch18}, which are in the image of the map from $MU_A^*$.

By direct observation and by Lemma \ref{KU_Surj}, both $MU_A^* \otimes_{MU^*} KU^*$ and $KU_A^*$ are quotients of $MU_A[\tilde{\beta}^{-1}]$. By the universal property of localizations, at most one such map of rings $MU_A^* \otimes_{MU^*} KU^* \rightarrow KU_A^*$ under $MU_A^*$ can exist. If $MU_A^* \otimes_{MU^*} KU^* \rightarrow KU_A^*$ is injective, then it is an isomorphism, and hence the induced map on any localization must be an isomorphism, hence injective. We will show this is not the case unless $A$ is trivial.

We will localize at the nontrivial Euler classes. More precisely, let $S$ denote the multiplicative monoid generated by the Euler classes $e_V$ where $V$ runs through nontrivial characters $A^* - \{ \textrm{triv} \}$. Let \[ \phi : (MU_A^* \otimes_{MU^*} KU^*)[S^{-1}] \rightarrow KU_A^*[S^-1] \] be the induced map.

From \cite[Proposition 2.11]{Hau22} and the isomorphism $L_A \cong MU_A^*$, we have \[ MU_A^*[S^{-1}] \cong MU^*[e_V^{\pm 1}, \gamma_i^V] \] Now \[ (MU_A^* \otimes_{MU^*} KU^*)[S^{-1}] \cong MU^*[e_V^{\pm 1}, \gamma_i^V] \otimes_{MU^*} KU^* \cong \Z[e_V^{\pm 1}, \gamma_i^V][\beta^{\pm 1}] \]
It must be the case that $\phi$ is the identity on $KU^*$, and $\phi$ takes Euler classes to Euler classes. In particular, the restriction of $\phi$ to $\Z[e_V^{\pm 1}][\beta^{\pm 1}]$ must be surjective. If $\phi$ were an isomorphism, then this restriction would also be injective, hence an isomorphism. Therefore $\phi$ can only be injective when \[ \Z[e_V^{\pm 1}, \gamma_i^V][\beta^{\pm 1}] = \Z[e_V^{\pm 1}][\beta^{\pm 1}] \] which only occurs when $A$ is trivial.
\end{proof}

\begin{corollary}
There is exactly one map of global group laws $MU_{(-)}^* \otimes_{MU^*} KU^* \rightarrow KU_{(-)}^*$ which is surjective at each level.
\end{corollary}

\begin{corollary}
The $A$-equivariant formal group laws determined by $MU_A \wedge_{MU} KU$ and $KU_A$ are not isomorphic, and $MU_{(-)}^* \otimes_{MU^*} KU^*$ and $KU_{(-)}^*$ are not isomorphic global group laws.
\end{corollary}

\begin{corollary}
\label{MayConjTori}
May's $MU_G \wedge_{MU} KU$ is not weakly equivalent to $KU_G$ as an $MU_G$-algebra for any nontrivial compact Lie group $G$.
\end{corollary}

\begin{proof}
If May's Conjecture \ref{MayConj} holds for some group $G$, then it must hold for every subgroup $H$ of $G$, because the restrictions of $MU_G \wedge_{MU} KU$ and $KU_G$ to a $H$-spectra respectively are $MU_H \wedge_{MU} KU$ and $KU_H$. The result follows from the fact that every nontrivial compact Lie group contains a nontrivial abelian subgroup.
\end{proof}

The following program could feasibly lead to an alternate proof of the failure of May's conjecture. There is a notion of height for an equivariant formal group law \cite{Str11} \cite{HM23} over appropriate fields, and it may be possible to state a general theorem computing this height for the examples arising from Proposition \ref{MayCpxOri}. After base-changing to a certain field, it may be possible to show that the $A$-equivariant formal group laws determined by $MU_A \wedge_{MU} KU$ and $KU_A$ have different heights.

\section{$BP_A$ and a lift of Quillen's idempotent}

Throughout this section, we fix a compact abelian Lie group $A$ and prime $p$, and implicitly localize everything at $p$. Localization is flat, so the $p$-localization of $MU_A$ is $A$-Landweber exact. In particular, any $p$-local cohomology theory determined by the $p$-localization of $MU_A$ is in fact determined by non-$p$-local $MU_A$, in the $A$-Landweber exactness sense.

Nonequivariantly, Quillen \cite{Qui69} constructs an idempotent natural transformation on $p$-local $MU$-cohomology whose image is $BP$-cohomology. In \cite{Ara79} this idempotent is shown to lift to the $C_2$-equivariant $MU$ known as Atiyah-Real $MU$, or $M \R$, and thus constructs a $C_2$-equivariant version of $BP$. Instead, we are interested in a distinct equivariant lift of the Quillen idempotent to the $MU_A$ for the purposes of constructing $BP_A$. 

The quick and dirty way to construct the cohomology theory $BP_A^*$ is to write $L_A \cong L_A \otimes_{L} L$, apply the Quillen idempotent in the second factor to obtain an idempotent on $L_A$, then observe that the image of this idempotent is a flat $MU_A^*$-module. The rest of this section is really an argument that this is a reasonable thing to do, as it has an equivariant formal group law interpretation, and agrees with another candidate construction for $BP_A$.

May constructs $MU_A \wedge_{MU} BP$ in \cite{May98} which in a later section we will show is naturally isomorphic to our $BP_A$. An advantage of May's construction is that it only requires $BP$ to be an $MU$-module (in other words it is relatively agnostic to the particular construction of $BP$). However, this makes opaque the connection of $BP_A$ to formal group laws. Conversely, the main advantage of our construction is that it makes it clear that $BP_A$ is $A$-Landweber exact via equivariant formal group law considerations.

Classically, the Quillen idempotent arises by studying the special property of formal group laws known as $p$-typicality. In particular, any formal group law over a $\Zp$-algebra has a canonical change-of-coordinate to a $p$-typical formal group law. A similar story can be told equivariantly, and the key idea is to define $p$-typicality as follows: an equivariant formal group law classified by $L_A \rightarrow k$ is $p$-typical if the nonequivariant formal group law classified by \[ L_{ \{ 1 \} } \xrightarrow{res_{ \{ 1 \} }^A} L_A \rightarrow k \] is in the usual sense.

\subsection{$p$-typicality for ordinary formal group laws}

The material in this subsection is adapted from \cite{Qui69} and \cite{COCTALOS}. Recall that we are implicitly working $p$-locally. We begin with the association of the group of curves to any formal group. This is the set of maps $\A \rightarrow X$ with the group operation induced by the group structure on $X$. This group has three kinds of natural operations: homothety, Verschiebung, and Frobenius.

Homothety is given by precomposing a given map $\A \rightarrow X$ with multiplication by a fixed $r \in k$, $\A \xrightarrow{r} \A \rightarrow X$. For each $n \in \mathbb{N}$, the $n$th Verschiebung $V_n$ is given by precomposing with the $n$th power map $\A \xrightarrow{x \mapsto x^n} \A \rightarrow X$. Lastly, the $n$th Frobenius is equivalently either given by the formula $(F_n g)(x) = \Sigma_{i=1,...,n}^G g(\sigma^i x^{1/n})$ or is given by the Verschiebung on the Pontryagin dual group of the group of curves.

\begin{definition}
An element $f$ of the group of curves is $p$-typical if $F_q f = 0$ for any prime $q \neq p$. Equivalently $F_n f = 0$ for any $n$ coprime to $q$.
\end{definition}

Note that any coordinate on a formal group specifies a curve. The associated formal group law is called $p$-typical if the curve determined in this way is a $p$-typical curve.

\begin{proposition}[Cartier]
Any formal group has a $p$-typical coordinate.
\end{proposition}

\begin{proof}
Starting with any coordinate on a formal group, construct a canonical change-of-coordinates to a $p$-typical coordinate as follows. First, define an operator on the group of curves by \[ \epsilon = \Sigma_{(m,p)=1} \frac{\mu(m)}{m} V_m F_m \] where $\mu(m)$ is the M\"{o}bius function, $V_m$ is the Verschiebung, and $F_m$ is the Frobenius. Next, given a coordinate $x$ (in other words, choosing an isomorphism from the formal group $X$ with $spf(k[[x]])$), consider $\epsilon x$. It can be verified that $\epsilon x$ specifies a new coordinate on $X$ which is $p$-typical.
\end{proof}

For the universal formal group law over $L$ (the one classified by the $p$-localization map from the non-$p$-local Lazard ring), the new formal group law obtained by passing to a $p$-typical coordinate is classified by a map $L \rightarrow L$ which is idempotent. Furthermore, the coordinate of a formal group law classified by $L \rightarrow k$ is $p$-typical if and only if the classification map factors through this idempotent.

Classically, Quillen uses this idempotent $L \rightarrow L$, which we call the Quillen idempotent, to provide a construction of $BP$. Specifically, he constructs a natural transformation from $MU^*(-)$ to itself which is idempotent (by slight abuse we also call this the Quillen idempotent), and whose image is $BP^*(-)$. This natural transformation is constructed by appealing to Yoneda's lemma: natural transformations from $MU^*(-)$ to itself are totally determined by the image of the identity in $MU^* MU$. By the splitting principle, it is possible to specify this image of the identity merely be specifying what happens to the first Chern class of the canonical bundle over $\C P^\infty$. The first Chern class is precisely the coordinate of the universal formal group law over $L$, so $p$-typicalization gives a new element of $MU^*(\C P^\infty)$.

Alternatively, we may construct the Quillen idempotent $MU^*(-) \rightarrow MU^*(-)$ as follows. The image $BP^*$ of the idempotent $L \rightarrow L$ is an $L$-module summand of $L$, and hence is projective and therefore flat over $L \cong MU^*$. Thus $MU^*(-) \otimes_{MU^*} BP^* =: BP^*(-)$ is a cohomology theory. By \cite[Theorem 4]{Qui69} this returns the same cohomology theory as the construction in the above paragraph. This is the perspective which we will generalize to the equivariant case.

\subsection{$p$-typicality for equivariant formal group laws}

Recall that any $A$-equivariant formal group law classified by $L_A \rightarrow k$ produces a nonequivariant formal group law classified by \[ L \xrightarrow{res_A^{ \{ 1 \} }} L_A \rightarrow k. \] The author is grateful to Neil Strickland for suggesting that $p$-typicality of an equivariant formal group law should be defined in the following way.

\begin{definition}[Strickland]
An $A$-equivariant formal group law is $p$-typical if the underlying nonequivariant formal group is $p$-typical.
\end{definition}

In the geometric approach to formal group laws of \cite{Str11}, this definition has the following intepretation. First, recall that an $A$-equivariant formal group is a formal scheme $X$ equipped with an equivariant map $A^* \rightarrow X$. $X$ contains a subscheme $\widehat{X}$, the infinitesimal part of $X$, which is an ordinary formal group.

Extracting $\widehat{X}$ from $X$ corresponds to extracting a nonequivariant formal group law from an $A$-equivariant formal group law classified by $L_A \rightarrow k$ by precomposing with $res_A^{ \{ 1 \} } : L_{ \{ 1 \} } \rightarrow L_A$. A coordinate $x$ on $X$ restricts to a coordinate $\widehat{x}$ on $\widehat{X}$, and we say that $x$ is $p$-typical if $\widehat{x}$ specifies a $p$-typical curve on $\widehat{X}$. 

Our definition of $p$-typicality therefore fails to capture information about the behavior of the original coordinate $x$ away from the infinitesimal part. It would be interesting to determine if there is an alternative definition of $p$-typicality which captures this information, particularly because it feasibly could give rise to another $A$-equivariantly complex oriented $A$-spectrum whose underlying spectrum is $BP$.

\begin{proposition}
There is a canonical change-of-coordinates rendering any given equivariant formal group law $p$-typical.
\end{proposition}

\begin{proof}
First, recall that the equivariant Lazard ring $L_A$ classifying $A$-equivariant formal group laws is naturally isomorphic to $MU_A^*$. Now note that functoriality of $MU_A^*$ in $A$ allows us to realize $MU_A^*$ as an $MU^*$-algebra. Better yet, there are ring maps $MU^* \rightarrow MU_A^* \rightarrow MU^*$ whose composition is the identity.

Next, we can apply the Quillen idempotent on coefficient rings \[ \xi : MU^* \rightarrow MU^* \] to obtain a map \[ \xi_A : MU_A^* \cong MU_A^* \otimes_{MU^*} MU^* \rightarrow MU_A^* \otimes_{MU^*} MU^* \cong MU_A^* \] whose image is $MU_A^* \otimes_{MU^*} BP^*$. Clearly this map restricts to the Quillen idempotent $\xi$ on $MU^*$. If we have an $A$-equivariant formal group law classified by $L_A \rightarrow k$, then precomposing with $\xi_A$ induces the desired change-of-coordinates.
\end{proof}

\begin{proposition}
\label{BP_A}
The $p$-typicalization of the coordinate of an equivariant formal group law defines an equivariant lift of Quillen's idempotent. The image of this idempotent is a cohomology theory represented by an $A$-spectrum $BP_A$ whose underlying nonequivariant spectrum is $BP$.
\end{proposition}

\begin{proof}
The $p$-typicalization classification map $\xi_A : MU_A^* \rightarrow MU_A^*$ exhibits $BP_A^*$ as an $MU_A^*$-module summand; in particular it is projective, hence flat over $MU_A^*$. The composition of flat ring maps is flat, so $BP_A^*$ is flat over $MU_A^*$. In particular, $MU_A^*(-) \otimes_{MU_A^*} BP_A^*$ is a cohomology theory with representing $A$-spectrum $BP_A$.

Next, note that $BP_A^*(-)$ is a summand of $MU_A^*(-)$ via the factorization of \[ Id \otimes \xi_A : MU_A^*(-) \otimes_{MU_A^*} MU_A^* \rightarrow MU_A^*(-) \otimes_{MU_A^*} MU_A^* \] as a surjection onto its image $BP_A^*(-)$ followed by inclusion of the image. Now, since the nonequivariant map underlying $\xi_A$ is $\xi$, it follows that the nonequivariant natural transformation underlying $Id \otimes \xi_A$ is the Quillen idempotent $Id \otimes \xi$.
\end{proof}

With the recent work of Hausmann and Meier \cite{HM23} on the invariant prime ideals of the Hopf algebra $( (MU_A)_*, (MU_A)_* MU_A ) \cong (L_A,S_A)$, which represents the groupoid of $A$-equivariant formal group laws, it stands to reason that explicit computations of $BP_A^*$ in the style of \cite{Sin01} \cite{Str01} \cite{AK16} \cite{KL21} are within reach. Generators for $MU_A^*$ clearly give generators for $BP_A^*$, but presumably there is a more efficient presentation, perhaps one that exhibits $BP_A^*$ as a deformation of $BP^*$.

We can certainly form the Hopf algebroid $((BP_A)_*, (BP_A)_*(BP_A))$, and it seems likely that it classifies $p$-typical equivariant formal group laws and their isomorphisms. In this case, one may guess that the invariant prime ideals of $( (BP_A)_*, (BP_A)_* BP_A )$ correspond to those invariant prime ideals of $( L_A, S_A )$ in exactly the same way that the invariant prime ideals of $(BP_*, BP_* BP)$ correspond to those of $(MU_*, MU_* MU)$.

It would be desirable to have an alternate description of the Quillen idempotent which more closely mimics Quillen's original construction in terms of Chern classes. In the remainder of this subsection, we outline a possible approach up to a major hurdle. 

The first Chern class of $\C P^\infty$ is equivalently the Thom class of the canonical bundle over $\C P^\infty$, using the isomorphism of the Thom space of this bundle with $\C P^\infty$ itself. A similar identification holds for complex projective space of a complete $A$-universe.

Since $MU_A^*$ determines the universal cohomology theory with Thom classes with respect to a certain choice of Thom classes, one could try to define the Quillen idempotent by using $p$-typicality to choose new Thom classes which satisfy the appropriate axioms. The main issue is then that it is not immediately clear how to choose Thom classes for equivariant vector bundles which are not line bundles.

Using the equivariant Schubert cell decomposition of equivariant Grasmannians in the next section along with complex stability of $MU_A^*$, the resulting spectral sequence is additively isomorphic to the Atiyah-Hirzebruch spectral sequence computing $MU^*(BU(n))$. However, the multiplicative structure in the equivariant case is very opaque, as even in the case of $BU(1) \cong \C P ( \mathcal{U}_A )$, the resulting cohomology ring is not necessarily isomorphic to formal power series on a single variable.

\subsection{Properties of $BP_A$}

May \cite{May98} has constructed another object which could reasonably be called $BP_A$, namely $MU_A \wedge_{MU} BP$. We will show in this section that our $BP_A$ is naturally isomorphic to $MU_A \wedge_{MU} BP$. Until then, we will use $BP_A$ to refer only to our previous construction of $BP_A$ and $MU_A \wedge_{MU} BP$ to refer only to May's construction. Furthermore, we will record the fact that $BP_A$ is $A$-Landweber exact.

May's construction makes sense for general compact Lie groups, not necessarily abelian. However, only for abelian groups is there a natural isomorphism $(MU_A \wedge_{MU} BP)^* \cong MU_A^* \otimes_{MU^*} BP^*$ (in general there is merely a spectral sequence relating these). To show that our $BP_A$ is the same as May's, we will show that there is a natural transformation inducing an isomorphism on homotopy groups. We will start by describing some of the structure enjoyed by $BP_A^*$.

\begin{proposition}
The rings $BP_A^*$ are contravariantly functorial in $A$. For $A$ a torus, the rings $BP_A^*$ assemble into a global group law.
\end{proposition}

\begin{proof}
Since $p$-typicality is detected by restricting to the trivial group, the maps $\xi_A$ commute with restriction maps along group homomorphisms. Thus the collection of subrings $BP_A^*$ are preserved by the restriction maps. By construction there is a map of $Tori$-algebras $\mathbf{L} \rightarrow BP_{(-)}^*$. Since $BP^*$ is a flat $MU^*$-module, tensoring over $MU^*$ with $BP^*$ preserves short exact sequences, particularly the short exact sequences which are required to exist for a global group law.
\end{proof}

\begin{proposition}
\[ BP_A^*(-) \cong MU_A^*(-) \otimes_{MU^*} BP^* \cong MU_A^*(-) \otimes_{MU_A^*} BP_A^*. \] In particular, $BP_A$ is $A$-Landweber exact.
\end{proposition}

\begin{proof}
The first isomorphism comes directly from the construction of $BP_A^*(-)$. The second isomorphism follows from $BP_A^* \cong MU_A^* \otimes_{MU^*} BP^*$. $A$-Landweber exactness of $BP_A$ then follows from flatness of $p$-localization.
\end{proof}

\begin{theorem}
\label{BP_A_Identification}
The $BP_A$ of Proposition \ref{BP_A} is weakly equivalent to May's $BP_A$.
\end{theorem}

\begin{proof}
By \cite{CM15} the Quillen idempotent thought of as a map of spectra $MU \rightarrow BP$ is $\mathbb{E}_2$, in particular it exhibits $BP$ as an $MU$-algebra. By Proposition \ref{MayCpxOri}, $MU_A \wedge_{MU} BP$ is $A$-equivariantly complex oriented, and the associated $A$-equivariant formal group law is classified by the base change $\xi_A$ of the Quillen idempotent $\xi$.

It now immediately follows from Theorem \ref{Uniqueness} that $BP_A$ and $MU_A \wedge_{MU} BP$ are weakly homotopy equivalent $A$-spectra.
\end{proof}

It is interesting that our construction of $BP_A$ agrees with $MU_A \wedge_{MU} BP$ but $KU_A$ does not agree with $MU_A \wedge_{MU} KU$. In some sense, this is because $MU_A \wedge_{MU} KU$ is "bigger" than $KU_A$. A stronger notion of $p$-typicality, perhaps one capturing the behavior of the coordinate away from the infinitesimal part of an equivariant formal group law, would have a smaller classifying ring. In particular, if such a definition did lead to an alternate construction of "equivariant Brown-Peterson spectra", then this new construction would likely be smaller than $MU_A \wedge_{MU} BP$ in a similar manner.

We can still define $BP_G$ for nonabelian compact Lie groups $G$ by base-changing the Quillen idempotent. However, it is no longer clear that this recovers May's construction of $BP_G$. Furthermore, since the theory of equivariant formal group laws currently only exists for compact abelian Lie groups, there is no evident equivariant formal group law interpretation of $BP_G$. It is a recent result of S. Kriz \cite{Kri22} that there exists a finite nonabelian group $G$ such that $(MU_G)_*$ is not a flat $MU_*$-module concentrated in even degrees. It would be interesting to know if this result extends to $BP_G$.

It would furthermore be desirable to know that there is a global spectrum $\mathbf{BP}$ which gives rise to the global group law $BP_A^*$ (Conjecture \ref{Global_Existence}). Our usual trick of base-changing the Quillen idempotent $\xi$ does not seem to work, because there is no obvious notion of "global homotopy groups". We could instead attempt to appeal to a version of Brown representability using Mackey functor structures everywhere instead of abelian group structures, but there is again an obstruction: global group laws do not visibly encode the data of transfers. However, It is possible that transfers are encoded more subtly, as there is a relationship between Euler classes and transfers for tori, and additionally the universal global group law $\underline{\pi}_*(\mathbf{MU})$ naturally has the structure of transfers.

\section{Phantom maps}

Fix a compact abelian Lie group $A$. To the author's knowledge, it was first observed by Hopkins that there are no phantom maps between spectra arising from Landweber exact formal group laws. With a few modifications, the proof from Lurie's lecture notes \cite{Lur10} goes through equivariantly.

First, we give our combination theorem/definition describing phantom maps which is the $A$-equivariant analog of \cite[Lemma 5]{Lur10}. The proof is identical; it uses Spanier-Whitehead duality, and the fact that every spectrum is a filtered colimit of suspension spectra or of finite spectra.

\begin{lemma}
\label{LurieLemma5}
For a map $f : E \rightarrow E'$, the following are equivalent:
\begin{enumerate}
\item $f$ induces the zero map on homology theories of $A$-spaces.
\item $f$ induces the zero map on homology theories of $A$-spectra.
\item $f$ induces the zero map on homology theories of finite $A$-spectra
\item $f$ induces the zero map on cohomology theories of finite $A$-spectra.
\item For every finite spectrum $X$ and map $g : X \rightarrow E$, the composition $f \circ g : X \rightarrow E'$ is nullhomotopic.
\end{enumerate}
If any of these equivalent conditions hold, we call $f$ a phantom map
\end{lemma}

\subsection{Reduction to complex generated $A$-spectra}

The essential part of the argument is the following. Landweber exact spectra $E$ are "determined" by $MU$, which is ultimately built up out of finite complexes with only even cells. For any "sufficiently even" spectrum $E'$, a spectral sequence argument shows that the $E'$-cohomology of a finite even complex is concentrated in even degrees. After this, the argument is essentially formal.

We will start by defining an equivariant version of "finite CW-spectrum with only even cells". The motivating examples are equivariant complex Grassmannians and their Thom spaces. To start with, recall that the (nonequivariant) Grassmannians $Gr(k,\C^n)$ have a Schubert cell decomposition associated to any complete flag $\{ F_i \}$ of $\C^n$. Namely, for any nondecreasing sequence $w_0, w_1, ..., w_n$ with $w_0 = 0$, $w_n = k$, and $w_{i+1} - w_i \leq 1$, there is a cell whose interior consists of subspaces $W$ such that $dim(W \cap F_i) = w_i$.

\begin{lemma}
For any finite-dimensional complex $A$-representation $V$, consider the $A$-space $Gr(k,V)$ of $k$-dimensional complex subspaces (not necessarily subrepresentations) of $V$. Then there exists a complete flag such that the corresponding Schubert cell decomposition consists of cells which are $A$-homeomorphic to complex $A$-representations.
\end{lemma}

\begin{proof}
Since $A$ is compact abelian Lie, $V$ decomposes as a direct sum of one-dimensional subrepresentations, $V \cong \oplus_i V_i$. Then take a complete flag to be $F_i = \oplus_{j=1}^i V_i$. Given a sequence $\{ w_i \}$, we can form a new sequence $\sigma(1),...,\sigma(k)$ which record the "jumps", namely $w_{\sigma(i)} - w_{\sigma(i-1)} = 1$.

This flag determines a basis $B_V$ of $V$, and if we have a basis $B_W$ of a subspace $W$, then we can form a matrix with $k$ columns and $n$ rows by writing the elements of $B_W$ in terms of $B_V$. Now we can column reduce, after which the pivot in column $j$ is in row $\sigma(j)$. All entries above this pivot are zero, and all entries below the pivot are parameters describing $W$.

By construction of $B_V$, the action of $A$ is by diagonal $n$ by $n$ matrices, multiplying on the left of our matrix describing $W$. For $g \in A$, let $\lambda_i(g)$ denote the $i$th diagonal entry of the corresponding matrix. Then $g$ acts on $W$ by multiplying the entry in row $i$, column $j$ below a pivot to $\lambda_{\sigma(j)}(g)^{-1} \lambda_i(g)$. Therefore the open cell corresponding to the sequence $\{ w_i \}$ is $A$-homeomorphic to the complex $A$-representation
\[ \oplus_{s=1}^k \left( V_{\sigma(s)}^{-1} \otimes \left( \oplus_{r = \sigma(s) + 1 \textrm{, } r \neq \sigma(i)}^{n} V_r \right) \right) \]
\end{proof}

\begin{lemma}
Let $X$ be an $A$ space which has a finite cell decomposition by cells $A$-homeomorphic to complex $A$-representations, and let $E$ be a complex $A$-vector bundle over $X$. Then the Thom space of $E$ has a finite cell decomposition by cells $A$-homeomorphic to complex $A$-representations.
\end{lemma}

\begin{proof}
Let $V$ be a complex $A$-representation homeomorphic to an open cell of $X$. Then $V$ is $A$-equivariantly contractible, hence the restriction of $E$ to $V$ is trivial, ie $A$-homeomorphic to the bundle $V \times W$ for some finite dimensional complex $A$-representation $W$. From here, the proof proceeds as in the nonequivariant case.
\end{proof}

\begin{definition}
A finite complex $A$-spectrum is an $A$-spectrum which has a finite cell decomposition by cells which are complex $A$-representations. If $A$ is the trivial group, this is the notion of a finite even spectrum.
\end{definition}

Rephrasing the above lemmas, we have that the suspension $A$-spectra of $Gr(k,V)$ and $Th(Gr(k,V))$ are finite complex $A$-spectra. Now we can begin to mimic the argument presented in \cite{Lur10}, with the only other modification being a spectral sequence argument. To start with, we extract a property of $A$-Landweber exact spectra.

\begin{definition}
A spectrum $E$ is complex generated if every map from a finite spectrum $X$ to $E$ factors through a finite complex spectrum.
\end{definition}

\begin{lemma}
\label{cpxgen}
Any $A$-Landweber exact spectrum $E$ is complex generated.
\end{lemma}

\begin{proof}
The proof is essentially identical to the argument \cite{Lur10}, except that we are working cohomologically rather than homologically. For the reader's convenience we reproduce it here.

Begin with a map $f : X \rightarrow E$ where $X$ is a finite spectrum. We can regard $f$ as an element of equivariant cohomology: $f \in [X,E]_0^A = E^0(X) \cong MU_A^0(X) \otimes_{MU_A^*} E^*$. As an element of the right-hand side, write $f$ as $\Sigma c_i m_i$ with $c_i \in MU_A^{d_i}(X)$ and $m_i \in E^{-d_i}$, so that $f$ factors as \[ X \xrightarrow{ \oplus c_i} \oplus \Sigma^{d_i} MU_A \xrightarrow{\oplus m_i} E. \] By evenness of $MU_A^*$ (cf \cite{Lof73} \cite{Com96}) and evenness of $E^*$, each $d_i$ is even, so it is enough to show that $MU_A$ is complex generated.

Next, $MU_A$ may be written as a colimit over a diagram indexed by Thom spaces of canonical bundles over Grassmannians $Gr(k,V)$ (cf the construction of global $\mathbf{MU}$ in \cite[Section 6.1]{Sch18}). We have already shown that each of these Thom complexes is a finite complex $A$-spectrum. Any map from a finite spectrum factors through one of the terms in the colimit diagram, so the result follows.
\end{proof}

\subsection{Phantom maps for complex generated $A$-spectra}

To show there are no phantom maps between $A$-Landweber exact spectra, the only condition on the domain that we need is that it is complex generated. Nonequivariantly, we need to assume that the codomain has homotopy groups concentrated in even degrees, in order to run a spectral sequence argument. This spectral sequence arises from the skeletal filtration on a finite even spectrum. 

Equivariantly, we have replaced "finite even spectrum" with "finite complex spectrum". In particular, the associated graded pieces are wedges of representation spheres. Thus, in order to exploit the same spectral sequence argument, we further need to assume that the codomain is complex stable. In other words, $E'^*(S^V) \cong E'^*(S^{dim_\R(V)})$. Any complex orientable $A$-spectrum is complex stable, so $A$-Landweber exact spectra are complex stable.

\begin{lemma}
\label{SubForAHSS}
Let $X$ be a finite complex spectrum and $E'$ a complex stable spectrum whose homotopy groups are concentrated in even degrees. Then $E'^*(X)$ is also concentrated in even degrees.
\end{lemma}

\begin{proof}
By assumption on $X$ we have a filtration $* = X_0 \subset X_1 \subset ... \subset X_n = X$ such that $X_{2i} = X_{2i+1}$ and the cofiber of $X_{2i-1} \rightarrow X_{2i}$ is a wedge of finitely many complex representation spheres. Taking the mapping spectrum $map(-,E')$ of this filtration yields a filtration on $map(X,E')$, and the associated spectral sequence converges strongly: \[ \pi_{p+q} \left( cofib \left( map \left(X_{p-1}, E' \right) \rightarrow map \left( X_p, E' \right) \right) \right) \implies \pi_{p+q}(map(X,E')) = E'^{p+q}(X) \]

Now we must identify the $E_1$-page. For odd $p$, $X_{p-1} = X_p$, so the cofiber is zero. For even $p$, mapping spectra commute with cofibers in the first variable, and the cofiber of $X_{p-1} \rightarrow X_p$ is a wedge of complex representation spheres, $\vee_i \mathbb{S}^{V_i}$. We compute $\pi_{p+q}( map(\vee_i \mathbb{S}^{V_i} , E') \cong \oplus_i E'^{p+q}(S^{V_i}) \cong \oplus_i E'^{p+q-dim_\R(V_i)}$, using complex stability of $E'$.

Since $p$ and $dim_\R(V_i)$ are even, and $E'^*$ is concentrated in even degrees, we see that if $q$ is odd, the corresponding terms of the $E_1$-page are zero. Since we already observed that $E_1^{p,q}$ is zero when $p$ is odd, we see that $E_1^{p,q}$ is nonzero only when both $p$ and $q$ are even. Thus $E'^*(X)$ is concentrated in even degrees, as desired.
\end{proof}

We may now proceed to the main technical result of this section.

\begin{theorem}
Let $E$ be a complex generated $A$-spectrum and $E'$ a complex stable $A$-spectrum whose homotopy groups are concentrated in even degrees. Then every phantom map from $E$ to $E'$ is nullhomotopic.
\end{theorem}

\begin{proof}
This proof is again nearly identical to that appearing in \cite{Lur10}. We include it for the reader's convenience.

Let $A$ be a set of representatives for all homotopy equivalence classes of maps $X_\alpha \rightarrow E$, for $X_\alpha$ a finite complex spectrum. Then we can form the fiber sequence \[ K \rightarrow \oplus_\alpha X_\alpha \xrightarrow{u} E \rightarrow \Sigma K \]

Now let $p : E \rightarrow E'$ be a phantom map. By criterion (5) of Lemma \ref{LurieLemma5}, $p \circ u$ is nullhomotopic, so $p$ factors through $\Sigma K$. So it is sufficient to show that every map $\Sigma K \rightarrow E'$ is nullhomotopic, or equivalently that $E'^{-1}(K)$ is zero. By Lemma \ref{SubForAHSS}, this will follow from showing that $K$ is a retract of a direct sum of finite complex spectra.

Let $B$ be the collection of triples $(\alpha, \alpha', f)$ with $\alpha, \alpha' \in A$ and $f$ ranging through homotopy equivalence classes of maps $X_\alpha \rightarrow X_{\alpha'}$ of spaces over $E$. For each such triple $\beta = (\alpha, \alpha', f)$, set $Y_\beta = X_\alpha$. We will show that $K$ is a retract of $\oplus_\beta Y_\beta$.

There is a map $\oplus_\beta Y_\beta \rightarrow \oplus_\alpha X_\alpha$ given by the difference of $Y_\beta = X_\alpha \hookrightarrow \oplus_\alpha X_\alpha$ and $Y_\beta = X_\alpha \xrightarrow{f} X_{\alpha'} \hookrightarrow \oplus_\alpha X_\alpha$. Take $F$ to be the cofiber of $\phi$, so that we have a map of fiber sequences \[ \begin{tikzcd} \oplus_\beta Y_\beta \arrow{r} \arrow{d} & \oplus_\alpha X_\alpha \arrow{r} \arrow{d} & F \arrow{d} \\
K \arrow{r} & \oplus_\alpha X_\alpha \arrow{r} & E \end{tikzcd} \]

Next, we want to construct a map going the opposite direction, $q : E \rightarrow F$; we will construct this by defining a natural transformation of homology theories of finite spectra. Passing to Spanier-Whitehead duals, it is sufficient to define $q(f) : X \rightarrow F$ for every map $f : X \rightarrow E$ (with $X$ an arbitrary finite spectrum). Since $E$ is complex generated, $f : X \rightarrow E$ factors through some $X_{\alpha'}$ as $X \xrightarrow{f'} X_{\alpha'} \rightarrow E$ for some $\alpha' \in A$. Define $q(f)$ as the composite $X \xrightarrow{f'} X_{\alpha'} \hookrightarrow \oplus_\alpha X_\alpha \rightarrow F$.

To show $q(f)$ is well-defined, suppose we have another factorization of $f$ as $X \xrightarrow{f''} X_{\alpha''} \rightarrow E$. Taking $Y$ to be the pushout $X_{\alpha'} \sqcup_X X_{\alpha''}$, $Y$ is a finite spectrum, and we have a canonical map $Y \rightarrow E$. $E$ is complex generated, so we get a composition $Y \xrightarrow{g} X_{\alpha} \rightarrow E$ for some $\alpha \in A$. Defining $h'$ to be $X_{\alpha'} \rightarrow Y \xrightarrow{g} X_\alpha$ and defining $h''$ similarly, we get two elements of $B$, $(\alpha', \alpha, h')$, and $(\alpha'', \alpha, h'')$. Thus the constructions of $q(f)$ from either factorization of $f$ are both equal to \[ X \rightarrow Y \xrightarrow{g} X_\alpha \hookrightarrow \oplus_{\alpha} X_\alpha \rightarrow F \]

Now that we have constructed our map on homology, we get a map of $G$-spectra $E \rightarrow F$. This gives us a diagram of fiber sequences \[ \begin{tikzcd}
K \arrow{r} \arrow{d} & \oplus_\alpha X_\alpha \arrow{d} \arrow{r} & E \arrow{d} \\
\oplus_\beta Y_\beta \arrow{d} \arrow{r} & \oplus_\alpha X_\alpha \arrow{d} \arrow{r} & F \arrow{d} \\
K \arrow{r} & \oplus_\alpha X_\alpha \arrow{r} & E 
\end{tikzcd} \]
By contruction the right vertical map induces the identity on homology, hence is an equivalence. Thus the left vertical composition is an equivalence. In particular, $K$ is a retract of $\oplus_\beta Y_\beta$, which is a direct sum of finite complex spectra, as desired.
\end{proof}

Putting this all together, we have:

\begin{theorem}
Let $E$ and $E'$ be $A$-Landweber exact spectra. Then every phantom map from $E$ to $E'$ is nullhomotopic.
\end{theorem}

\begin{proof}
By Lemma \ref{cpxgen} $A$-Landweber exact spectra are complex generated. $A$-Landweber exact spectra have homotopy groups concentrated in even degrees, and by $A$-equivariant complex orientability, they are complex stable.
\end{proof}

\begin{corollary}
Any phantom map with codomain and domain chosen from the following list of genuine $A$-spectra, or their localizations at a prime $p$, is nullhomotopic: $MU_A$, $KU_A$, and $BP_A$.
\end{corollary}

\section{Acknowledgments}

The author would like to thank Mike Hill and Peter May for serving consecutively as thesis advisor during the gestation of the ideas contained within and for answering questions on the basic principles of equivariant homotopy theory. Also, the author would like to thank Lennart Meier, Markus Hausmann, Stefan Schwede, and Neil Strickland for helpful conversations regarding the content of this paper.

Additionally, the author thanks Jack Carlisle for catching a mistake in an early draft; the fix for this mistake (using base change to define equivariant $p$-typicalization rather than an explicit power series formula) led to the ideas necessary to interface with May's conjecture. A portion of these results will likely appear in the author's PhD thesis, the main result of which the author hopes to be an equivariant Landweber exact functor theorem.

\bibliographystyle{alpha}
\bibliography{NPM}

\end{document}